\documentclass[01pt]{amsart}
\usepackage{amssymb,amsmath,amsthm,mathrsfs,graphicx,color,fullpage,tikz}

\newcommand{\Tr}{{\rm Tr}}

\newcommand{\PGL}{{\rm PGL}}

\newcommand{\G}{{\Gamma}}
\newcommand{\g}{{\gamma}}
\renewcommand{\dot}{{}}
\newcommand{\Ind}{{\rm Ind}}

\newcommand{\C}{\mathbb{C}}
\newcommand{\oo}{\mathcal{O}}
\renewcommand{\P}{\mathcal{P}}
\usepackage{multirow}
\newtheorem{thm}{Theorem}[subsection]
\newtheorem{theorem}[thm]{Theorem}

\newtheorem{corollary}[thm]{Corollary}
\newtheorem{proposition}[thm]{Proposition}
\numberwithin{equation}{section} \theoremstyle{definition}

\newtheorem*{remark}{Remark}

\begin{document}

\def \Li {\textcolor{red}}
\def \Kang {\textcolor{blue}}

\title[Artin $L$-functions on PGL$_3$]{Artin $L$-functions on PGL$_3$}
\author{Ming-Hsuan Kang and Wen-Ching Winnie Li}
\address{Ming-Hsuan Kang\\ Department of Applied Mathematics\\National Chiao-Tung University\\
Hsinchu, Taiwan} \email{\tt mhkang@nctu.edu.tw}
\address{Wen-Ching Winnie Li\\ Department of Mathematics\\ The Pennsylvania State University\\
University Park, PA 16802 and\\National Center for Theoretical Sciences\\ National Tsing Hua University\\Hsinchu, Taiwan} \email{\tt wli@math.psu.edu}
\thanks{The research of the first author is supported by the NSC grant 100-2115-M-009-008-MY2 and 102-2115-M-009 -005.
The research of the second author is partially supported by the NSF grant DMS-1101368.}


\begin{abstract}
We study Artin $L$-functions on a finite $2$-dimensional complex $X_\G$ arising from $\PGL_3$ attached to finite-dimensional representations $\rho$ of its fundamental group. Some key properties, such as 
 rationality, functional equation, and invariance under induction,  of these functions are proved. Moreover, using a cohomological argument, we establish a connection between 
the Artin $L$- functions and the $L$-function of $\rho$, extending the identity on zeta functions of $X_\G$ obtained in \cite{KL, KLW}.
\end{abstract}

\maketitle

\section{Introduction}

Let $F$ be a nonarchimedean local field with the ring of integers $\oo$ and a uniformizer $\pi$ such that its residue field $\kappa =\oo / \pi \oo$ has cardinality $q$. The building
attached to the group $G=\PGL_3(F)$ is a contractible 2-dimensional simplicial complex $X$, whose vertices are homothety classes of $\oo$ lattices of rank-$3$ in the $3$-dimensional vector space $F^3$. The group $G$ acts transitively on $X$ as automorphisms.

Fix a discrete cocompact torsion-free subgroup $\Gamma$ of $G$ so that $X_\G= \G \backslash X$ is a finite complex locally isomorphic to $X$. Then $X_\G$ has $X$ as its universal cover and $\G$ as its fundamental group.
Two kinds of zeta functions, the edge zeta function and the chamber zeta function, on $X_\G$ were considered in \cite{KL,KLW}. They extend the Ihara zeta functions for graphs to finite two-dimensional complexes. {Analogous to the two expressions for the Ihara zeta function in terms of the vertex adjacency operator by Ihara  \cite{Ih} and edge adjacency operator by Hashimoto \cite{Ha1}, the edge and chamber zeta functions for $X_\G$ are shown to be rational functions and they satisfy an identity involving operators on the vertices, edges and chambers, resembling the zeta function for a smooth irreducible projective surface defined over a finite field. Along a similar vein, the edge and chamber zeta functions for finite quotients of the building of another rank two group $Sp_4(F)$ have been investigated in \cite{FLW}, where rationality and identities for the zeta functions are also obtained, but they are more complicated due to the nature of the group.}

Let $\rho$ be a $d$-dimensional 
representation of $\G$ acting on the space $V_\rho$ over $\mathbb C$. In this paper we study the $i$-th Artin $L$-function of $X_\G$ attached to $\rho$ for $i=1, 2$ defined by
$$L_{i}(X_\G,\rho, u) =   \prod_{[\mathfrak{p}]}  \det\left( I- \rho(Frob_{[\mathfrak{p}]}) u^{l_A(\mathfrak{p})} \right)^{-1},$$
where $I$ is the identity $d \times d$ matrix, $[\mathfrak{p}]$, which plays the role of an $i$-dimensional prime, runs through all equivalence classes of primitive  uni-type closed $i$-dimensional geodesics $\mathfrak{p}$
in $X_\G$,  $l_A(\mathfrak{p})$ is the algebraic length of $\mathfrak{p}$, and  $Frob_{[\mathfrak{p}]}$ is a conjugacy class in $\G$ associated to the \lq\lq prime" $[\mathfrak{p}]$. See \S3.3 and \S4.4 for detailed definitions.  When $\rho$ is the trivial representation of $\G$, the above $L$-functions coincide with the edge and chamber zeta functions studied in \cite{KL, KLW}.

Denote by $N_i$ the number of $i$-dimensional simplices in $X_\G$. The Euler characteristic $\chi(X_\G)$ of the complex $X_\G$ is equal to $N_0-N_1+N_2$.
We summarize the main properties of these Artin $L$-functions.

\begin{thm}\label{z1}
$L_{1}(X_\G,\rho, u)$ converges absolutely {for $|u|$ small enough} to a rational function of the form
$$L_{1}(X_\G,\rho, u)= \frac{1}{\det(I- A_E(\rho, u))}, $$
where $A_E(\rho, u)$ is an edge adjacency operator acting on a free $\C[u]$-module of rank $2dN_1$. Consequently, $L_{1}(X_\G,\rho, u)$ has a meromorphic continuation to the whole $u$-plane, with reciprocal equal to a polynomial of degree $2dN_1$.
\end{thm}

\begin{thm}\label{z2}
$L_{2}(X_\G,\rho, u)$ converges absolutely {for $|u|$ small enough} to a rational function of the form
$$L_{2}(X_\G,\rho, u)= \frac{1}{\det(I- A_C(\rho, u))}, $$
where $A_C(\rho, u)$ is a chamber adjacency operator acting on a free $\C[u]$-module of rank $3dN_2$. Consequently, $L_{2}(X_\G,\rho, u)$ has a meromorphic continuation to the whole $u$-plane, with reciprocal equal to a polynomial of degree $3dN_2$.
\end{thm}

These zeta functions encode the geometric information of $\G$. The spectral information of $\G$ is characterized by the (local) $L$-function which we now explain.
The Satake parameter attached to an irreducible unramified  representation $\sigma$ of $G$ is a semisimple conjugacy class $s(\sigma)$ in the complex dual group $\hat{G}(\mathbb C)\cong$ SL$_3(\C)$ of $G$. The  $L$-function of $\sigma$ attached to the standard representation of SL$_3(\C)$ is
$$ L(\sigma, u) = \det(I- s(\sigma)u)^{-1}.$$
Since $\G$ is discrete and cocompact, the induced representation $\Ind_\G^G \rho $ can be decomposed into a direct sum of irreducible subrepresentations. Define the (unramified) $L$-function of $\Ind_\G^G \rho$ to be
$$  L(\Ind_\G^G \rho, u) = \prod_{\sigma} L(\sigma, u)^{m(\sigma)}, $$
where $\sigma$ runs through all unramified irreducible representations of $G$ and $m(\sigma)$ is the multiplicity of $\sigma$ in  $\Ind_\G^G \rho$. It is shown in \S 2 that the reciprocal of  $L(\Ind_\G^G \rho, u)$ is a polynomial of degree $3dN_0$, and Proposition 4.3.1 shows its connection with the two vertex adjacency operators on $X_\G$.

The main purpose of this paper is to prove the following identity on $L$-functions.

\begin{theorem} \label{thm1}
With the above notation, for a $d$-dimensional representation $\rho$ of $\G$ we have
$$(1-u^3)^{\chi(X_\G) d}  L(\Ind_\G^G \rho , qu) = \frac{L_{1}(X_\G,\rho, u)}{L_2(X_\G,\rho, -u)}.$$
\end{theorem}
This theorem is proved by a cohomological argument. More precisely, we define a cochain complex $C^*$ 
whose $i$-th cochain group $\dot{C}^i(X_\G, \rho)$ for $0 \le i \le 2$ consists of $V_\rho \otimes_{\mathbb C} \mathbb C[u]$-valued functions on {\it pointed} $i$-simplicies on which the group $\G$ acts via $\rho$, and the coboundary maps are suitable deformations (involving the variable $u$) of the usual coboundary maps. See \S4.2 and \S5.1 for details.
We show in \S6 that there exist cochain endomorphisms $\Phi_{i}= \Phi_i(u)$ on $C^i(X_\G, \rho)$ for $i=0,1$ and 2 whose determinants interpret the $L$-functions introduced above:
\begin{theorem} \label{thm2} $ $
\begin{enumerate}
\item $\det( \Phi_0 ~|~C^0(X_\G, \rho))=L(\Ind_\G^G \rho , qu)^{-1}.$
\item $\det( \Phi_1 ~|~C^1(X_\G, \rho))=(1-u^3)^{dN_1}L_1(X_\G,\rho, u)^{-1}.$
\item $\det( \Phi_2 ~|~C^2(X_\G, \rho))=(1-u^3)^{2dN_2}L_2(X_\G,\rho, -u)^{-1}.$
\end{enumerate}
\end{theorem}

\noindent The desired identity in Theorem \ref{thm1} then follows from the fact that for each $i$, the cochain map $\Phi_i$ on $C^i(X_\G, \rho)$ is homotopically equivalent to the cochain map multiplication by $1-u^3$.

\begin{remark} The combinatorial Artin $L$-functions attached to representations were considered by Ihara in \cite{Ih}, Hashimoto in \cite{Ha2, Ha3}, Mizuno and Sato \cite{MS}, and Stark and Terras in \cite{ST2} for graphs. In the case of a finite connected undirected graph $Y$, the equivalence classes of primitive tailless closed geodesics in $Y$, that is, the \lq\lq primes\rq\rq for $Y$, naturally correspond to the conjugacy classes of nonidentity primitive elements in the fundamental group of $Y$. As pointed out in \cite{KL}, this is no longer the case for our $X_\G$. Namely there are more Frobenius conjugacy classes than conjugacy classes of nonidentity primitive elements in $\G$. Our results are generalizations of those of Hashimoto \cite{Ha2} from graphs to two-dimensional complexes, but our method is different from his. For trivial $\rho$, two different proofs are in the literature: the one in \cite{KL} results from counting the number of desired closed geodesics of given length, while in \cite{KLW} the identity is derived using representation theory. The cohomological method described above is a generalization of the approach by Bass \cite{Ba} and re-interpreted by Hoffman \cite{Hof} for graphs. In particular it provides a third proof of the identities on zeta functions established in \cite{KL, KLW}.
\end{remark}

Set
$$ \epsilon( \rho, u) = \left(1-\left(\frac{u}{q}\right)^3\right)^{dN_0/2}\bigg(1-(qu)^3\bigg)^{dN_0/2}.$$
In \S2.2 we show that there is a functional equation relating the $L$-function of $\Ind_\G^G \rho$ and that of $\Ind_\G^G \rho^*$.

\begin{theorem} \label{thm-functionalequation}
The following functional equation holds:
$$ \epsilon(\rho,\frac{1}{q u}) L(\Ind_\G^G \rho,\frac{1}{q u})= \epsilon(\rho^*, q u ) L(\Ind_\G^G \rho^*,qu).$$
Here $\rho^*$ is the contragredient representation of $\rho$.
\end{theorem}

Combined with Theorem \ref{thm1}, the above functional equation can be restated in terms of the quotient of the Artin $L$-functions.
\begin{theorem} \label{thm-functionalequation}
The following functional equation holds:
$$
\tilde{\epsilon}(\rho,\frac{1}{q^2u})
\frac{L_1(X_\G, \rho, \frac{1}{q^2 u})}{L_2(X_\G, \rho, -\frac{1}{q^2 u})}
= \tilde{\epsilon}(\rho^*,u) \frac{L_1(X_\G,\rho^*, u)}{L_2(X_\G, \rho^*, -u)}.$$
Here $ \tilde{\epsilon}(\rho,u) = \epsilon( \rho, q u) (1-u^3)^{-\chi(X_\Gamma)d}.$
\end{theorem}

 It follows immediately from the definition that for $i = 1, 2$ the Artin $L$-function decomposes into a product when the representation is a direct sum, that is,
 $$L_i(X_\G, \rho_1\oplus \rho_2, u) = L_i(X_\G, \rho_1, u)L_i(X_\G, \rho_2, u).$$ In \S7 we show that the Artin $L$-function is invariant under induction, just like the
 usual Artin $L$-functions attached to representations of the absolute Galois group of a number field.

\begin{theorem}\label{induction} 
Suppose $\rho$ is induced from a finite-dimensional representation $\rho'$ of a finite-index subgroup $\G'$ of $\G$. Let $X_{\G'} = \G' \backslash X$. Then for $i = 1, 2$ we have
$$ L_i(X_\G, \rho, u) = L_i(X_{\G'}, \rho', u).$$
\end{theorem}

The corresponding statement for graphs was proved by Hashimoto in \cite{Ha3}, where it was proved by counting closed geodesics, using definition of the Artin $L$-function. Our proof in \S7 compares the actions of the edge/chamber adjacency operators on $X_\G$ and $X_{\G'}$,  using Theorems \ref{z1} and \ref{z2}.

In particular, when $\rho'$ is the identity representation of a finite-index normal subgroup $\G'$, the induced representation $\Ind_{\G'}^{\G} 1$ decomposes into the direct sum $\oplus_{\sigma \in \widehat{\G/\G'}} m(\sigma) \sigma$, where $\widehat{\G/\G'}$ consists of all irreducible representations of the quotient group $\G/\G'$, and the multiplicity $m(\sigma)$ is equal to the degree of $\sigma$. Thus
$$ Z_i(X_{\G'}, u) = L_i(X_{\G'}, 1, u) = \prod_{\sigma \in \widehat{\G/\G'}} L_i(X_{\G}, \sigma, u)^{m(\sigma)}$$
for $i = 1, 2$. By Theorems \ref{z1} and \ref{z2}, the reciprocal of each $L_i$ above is a polynomial, hence we conclude

\begin{corollary} 
Let $\G'$ be a normal subgroup of $\G$ of finite index. Then for $i = 1$ and $2$,
$Z_i(X_\G, u)^{-1}$ divides $Z_i(X_{\G'}, u)^{-1}$.
\end{corollary}

The corresponding statement for graphs was proved in \cite{Ha2}.

\section{$L$-functions and functional equations} \label{lfunction}
\subsection{$L$-functions}
The group $K=\PGL_3(\oo)$ is the standard maximal compact subgroup of $G=\PGL_3(F)$.
The Hecke algebra $H(G,K )$ is generated by the following two Hecke operators:
$$ A_1 = K \left(\begin{smallmatrix} 1 & &\\ &1& \\ & & \pi \end{smallmatrix}\right)K \qquad \mbox{and} \qquad
 A_2 = K \left(\begin{smallmatrix} 1 & &\\ & \pi & \\ & & \pi \end{smallmatrix}\right)K.$$
The Satake isomorphism (\cite{Sat}) $\phi:H(G,K) \to \mathbb C[z_1,z_2,z_3]^{S_3}/(z_1z_2z_3-1)$ is characterized by
$$ \phi(A_1)=q(z_1+z_2+z_3) \qquad \mbox{and} \qquad \phi(A_2)=q(z_1 z_2 + z_2 z_3 + z_1 z_3).$$
For an unramified irreducible representation $(\sigma,V_\sigma)$ of $G$ with the Satake parameter $s(\sigma)$ equal to the conjugacy class of $\left(\begin{smallmatrix} \lambda_1 & &\\ & \lambda_2 & \\ & & \lambda_2 \end{smallmatrix}\right)$,
its $K$-fixed subspace $V_\sigma^K$ is one-dimensional on which $I-A_1 u + q A_2 u^2 - q^3 u^3 I$ acts as multiplication by the scalar
\begin{eqnarray*}
\det(I-A_1 u + q A_2 u^2 - q^3 u^3 I ~|~ V_\sigma^{K}) &=&
1-q(\lambda_1+\lambda_2+\lambda_3) u + q^2(\lambda_1 \lambda_2 + \lambda_2 \lambda_3 + \lambda_1 \lambda_3) u^2 - q^3 u^3\\
&=& \prod_{i=1}^3 (1-q \lambda_i u)=L(\sigma, qu)^{-1}.
\end{eqnarray*}
Here and thereafter $\det(A | W)$ denotes the determinant of the linear operator $A$ on the finite dimensional vector space $W$.
The induced representation
$$\Ind_\G^G \rho = \{ f:  G \to V_\rho 　\, | \, f(\g x) = \rho(\g) f(x), \, ~\text{for~all}~ \g \in \G ~\text{and} ~x \in G\}$$
decomposes into a direct sum of irreducible  representations $\sigma$ of $G$.
The total number of unramified $\sigma$'s, counting multiplicity, is equal to the dimension of $(\Ind_\G^G \rho)^K$. Hence
\begin{align*}
L(\Ind_\G^G \rho , q u) &=  \prod_{\sigma} L(\sigma, q u) = \prod_{\sigma} \det(I -A_1 u + q A_2 u^2 - q^3 u^3 I ~|~ V_\sigma^{K})^{-1}\\
& =
\det(I-A_1 u + q A_2 u^2 - q^3 u^3 I ~|~ (\Ind_\G^G \rho)^K)^{-1}.
\end{align*}
We record this in
\begin{proposition} \label{Lfunction1}
$$ L(\Ind_\G^G \rho , q u)=
\det(I-A_1 u + q A_2 u^2 - q^3 u^3 I ~|~ (\Ind_\G^G \rho)^K)^{-1}.$$
\end{proposition}

Note that the dimension of $(\Ind_\G^G \rho)^K$ is equal to the cardinality of the double cosets in $\G \backslash G /K$ times the dimension of $V_\rho$, that is, $N_0d$. Hence the denominator of  $L(\Ind_\G^G \rho , q u)$ is a polynomial of degree $3dN_0$.

\subsection{A functional equation} In this subsection we prove the functional equation satisfied by $L(\Ind_\G^G \rho , q u)$ as stated in Theorem \ref{thm-functionalequation}. Given an irreducible unramified representation $(\sigma,V)$ with the Satake parameter $s_\rho=\left(\begin{smallmatrix} \lambda_1 & & \\ & \lambda_2 & \\ & & \lambda_3 \end{smallmatrix}\right)$, its contragredient representation $(\sigma^*,V^*)$ has the Satake parameter equal to $s_{\rho}^{-1}$.
Since $\lambda_1\lambda_2\lambda_3=1$, we have
$$ A_1\big|_V = q(\lambda_1+\lambda_2+\lambda_3) =
q(\lambda_2^{-1}\lambda_3^{-1}+\lambda_1^{-1}\lambda_3^{-1}+\lambda_1^{-1}\lambda_2^{-1})
=A_2\big|_{V^*}$$
and
$$A_2\big|_V = q(\lambda_2\lambda_3+\lambda_1\lambda_3+\lambda_1\lambda_2) =q(\lambda_1^{-1}+\lambda_2^{-1}+\lambda_3^{-1}) = A_1\big|_{V^*}.$$
Therefore,
$$ L( \sigma^* , q u) = \det(1-A_2 u + q A_1 u^2 - q^3 u^3|~ V)^{-1}$$
and
\begin{align*}
L(\Ind_\G^G \rho,\frac{1}{q u})
&= \det\left(1-A_1\frac{1}{q^2 u} + A_2\frac{1}{q^3u^2} -\frac{1}{q^3u^3} ~|~(\Ind_\G^G \rho)^K\right)^{-1} \\
&= (-q^3u^3)^{d N_0}\det(1-A_2u + q A_1u^2-q^3 u^3 ~|~(\Ind_\G^G \rho)^K)^{-1} \\
&=  (-q^3u^3)^{d N_0} L( \Ind_\G^G \rho^* , q u) .
\end{align*}
Let
$$ \epsilon( \rho, u) = \left(1-\left(\frac{u}{q}\right)^3\right)^{dN_0/2}\bigg(1-(qu)^3\bigg)^{dN_0/2}.$$
It is easy to verify that
$$ \epsilon(\rho,\frac{1}{q u}) L(\Ind_\G^G \rho,\frac{1}{q u})= \epsilon(\rho^*,q u ) L(\Ind_\G^G \rho^*,qu),$$
which proves Theorem \ref{thm-functionalequation}.

\section{Paths and galleries on the simplicial complex $X$}

These were discussed in detail in \cite{KL}. In this section we recall them and set up notation to be used later.

\subsection{The building $X$ of $\PGL_3(F)$}

The vertices of the building $X$ of $\PGL_3(F)$ are the homothety classes of $\oo$-lattices $a$ in $F^3$.
Given an inclusion relation of lattices $a_1 \supsetneq a_2 \supsetneq \cdots \supsetneq a_r$, denote by $[ a_1 \supsetneq a_2 \supsetneq \cdots \supsetneq a_r]$ the homothety class of this relation. Hence the vertices of $X$, also called the (pointed) $0$-simplicies of $X$, are denoted by $[a]$. Two vertices $[a_1]$ and $[a_2]$ form an edge (or 1-simplex) $E = \{[a_1], [a_2]\}$ of $X$ if there exist representatives $a_1$ and $a_2$
so that $\pi^{-1} a_2 \supsetneq a_1 \supsetneq a_2 (\supsetneq \pi a_1)$.
In this case, $a_1/ a_2$ is a proper subspace of $\pi^{-1}a_2 / a_2 \cong (\mathbb F_q)^3$ with dimension
$$ |a_1/a_2|:= \dim_{\mathbb F_q} a_1/a_2 = 1 \mbox{ or } 2, \quad \mbox{ and} \quad  |a_2/\pi a_1| =
\dim_{\mathbb F_q} a_2/\pi a_1 = 3 - |a_1/a_2|. $$
To $E$, we associate two pointed edges: $[\pi^{-1} a_2 \supsetneq a_1 \supsetneq a_2]$ of type $|a_1/a_2|$, and $[ a_1 \supsetneq a_2 \supsetneq \pi a_1]$ of type $|a_2/\pi a_1|$.
Define the {\it algebraic length} of a pointed edge to be its type.
Three vertices, $[a_1], [a_2]$ and $[a_3]$ form a chamber (or 2-simplex) $C = \{[a_1], [a_2], [a_3]\}$ if there exist representatives $a_1, a_2$ and $a_3$ so that $a_1 \supsetneq a_2 \supsetneq a_3 \supsetneq \pi a_1 (\supsetneq \pi a_2 \supsetneq \pi a_3)$.  In this case, we associate to $C$ three pointed chambers $[\pi^{-1}a_3 \supsetneq a_1 \supsetneq a_2 \supsetneq a_3] $, $[a_1 \supsetneq a_2 \supsetneq a_3 \supsetneq \pi a_1]$, and $[a_2 \supsetneq a_3 \supsetneq \pi a_1 \supsetneq \pi a_2]$.
 The {\it algebraic length} of a pointed chamber $[a_1 \supsetneq a_2 \supsetneq a_3 \supsetneq \pi a_1] $ is defined to be the type of the pointed edge $[a_1 \supsetneq a_3 \supsetneq \pi a_1]$, which is always equal to 1.

An element $g \in G$ acts on the vertices of $X$ by sending $[a]$ to $[ga]$. It preserves edges and chambers, and hence $G$ acts on $X$ as automorphisms.
Note that $K$ is the stabilizer of the vertex represented by the lattice spanned by the standard basis of $F^3$.
 As $G$ acts transitively on vertices of $X$, the coset space $G/K$ parametrizes the vertices of $X$.
{Furthermore, $G$ also acts transitively on pointed edges and pointed chambers and these two sets can be parametrized by cosets of certain parahoric subgroup and Iwahoric subgroup of $G$, respectively. See \cite{KL} for details.}

\subsection{Out-neighbors}
The out-neighbors of a pointed edge $[\pi^{-1} a_2 \supsetneq a_1 \supsetneq a_2]$ of type $|a_1/a_2|$ 
are the pointed edges $[\pi^{-1} a_3 \supsetneq a_2 \supsetneq a_3]$ with type $|a_2/a_3| = |a_1/a_2|$ such that $[a_1], [a_2], [a_3]$ do not form a chamber.
In this case we have two relations 
$a_1 \supsetneq a_2 \supsetneq \pi a_1  \supsetneq \pi a_2$ and $a_1 \supsetneq a_2 \supsetneq a_3 \supsetneq \pi a_2$. The condition $|a_1/a_2|=|a_2/a_3|$ implies that one of $a_3/\pi a_2$ and $\pi a_1/\pi a_2$ is a one-dimensional subspace of $a_2/ \pi a_2 \cong \mathbb{F}_q^3$ and the other is two-dimensional. Denote by $N(e)$ the collection of out-neighbors of a pointed edge $e$.
Therefore, we obtain a criterion for out-neighbors of a pointed edge:
\begin{eqnarray} \label{outerneighbor1}
[\pi^{-1} a_3 \supsetneq a_2 \supsetneq a_3] \in N([\pi^{-1} a_2 \supsetneq a_1 \supsetneq a_2])
& \Leftrightarrow & |a_1/a_2|=|a_2/a_3|, ~ a_3 \not\supset \pi a_1,~ \pi a_1 \not\supset a_3  \\
& \Leftrightarrow & |a_1/a_2|=|a_2/a_3|,~ a_3+ \pi a_1 = a_2,  \nonumber
\end{eqnarray}
where $a_3+\pi a_1$ is the lattice generated by $a_3$ and $\pi a_1$. Observe that a pointed edge has  $q^2$ out-neighbors.

For a pointed chamber $c= [\pi^{-1}a_3 \supsetneq a_1 \supsetneq a_2 \supsetneq a_3]$, its out-neighbors are pointed chambers $[\pi^{-1}a_4 \supsetneq a_2 \supsetneq a_3 \supsetneq a_4]$ with $[a_4]\neq [a_1]$; denote the collection by $N(c)$.
In terms of lattices, we have
\begin{eqnarray}  \label{outerneighbor2}
[\pi^{-1}a_4 \supsetneq a_2 \supsetneq a_3 \supsetneq a_4] \in N([\pi^{-1}a_3 \supsetneq a_1 \supsetneq a_2 \supsetneq a_3]) & \Leftrightarrow & a_4 \neq \pi a_1 \,\, \Leftrightarrow \,\, a_4+\pi a_1 = a_3.
\end{eqnarray}
Hence a pointed chamber has $q$ out-neighbors.

\subsection{Paths and galleries} \label{typeonepath}
An {\it edge path} $\mathfrak{p}$ of $X$ is a sequence   $ e_1 \to e_2 \to \cdots \to e_n$ of pointed edges
in the 1-skeleton of $X$ such that $e_{i+1}$ is an out-neighbor of $e_i$ for $i= 1,..., n-1$; all pointed edges in $\mathfrak{p}$ have the same type $j$, equal to $1$ or $2$, called the type of the path.
We define the geometric length $l_G(\mathfrak{p})$ of $\mathfrak{p}$ to be $n$ and the algebraic length $l_A(\mathfrak{p})$ to be $jn$. \
Note that a path in $X$ is a directed \emph{straight} line segment in an apartment.

A {\it type $1$ gallery} $\mathfrak{g}$ in $X$ is a sequence of pointed chambers $c_1 \to \cdots \to c_n$ in $X$ so that $c_{i+1}$ is an out-neighbor of $c_i$ for $i= 1,..., n-1$. In other words, there exists a sequence of lattices $a_1 \supsetneq \cdots \supsetneq a_{n+2}$ so that $c_i = [\pi^{-1} a_{i+2} \supsetneq a_i \supsetneq a_{i+1} \supsetneq a_{i+2}]$ for $1 \le i \le n$. We define both the geometric length $l_G(\mathfrak{g})$ and the algebraic length $l_A(\mathfrak{g})$ of $\mathfrak{g}$ to be $n$. 
Geometrically a type one gallery is a directed \emph{straight} gallery in an apartment.

For convenience, a type $1$ gallery in $X$ is called a uni-type 2-dimensional geodesic, and an edge path contained in the 1-skeleton of $X$ is called a uni-type 1-dimensional geodesic.

\section{Artin $L$-functions attached to representations of $\G$}

\subsection{The finite quotient $X_\G$} \label{finitecomplex}

Let $\G$ be a discrete cocompact torsion-free subgroup of $G$ so that $X_\G := \G \backslash G$ is a finite simplicial complex locally isomorphic to $X$. Since $X$ is contractible, $\G$ is isomorphic to the fundamental group of $X_\G$.
Explicit constructions of such finite complexes can be found in \cite{Sar} for instance, in which the 1-skeleton of the complexes may be described as Cayley graphs on subgroups of $\PGL_3(\mathbb F_q)$ containing PSL$_3(\mathbb F_q)$.

Denote by  $X_i$ the set of pointed $i$-simplices of $X$ for $i=0, 1, 2$.
The group $\G$ acts freely and transitively on $X_i$ by left translation. 
Fix a choice of a subset $S_i$ of $X_i$ representing the orbit space $\G \backslash X_i$. Then the elements in $X_i$ can be labeled by $\G S_i$. Geometrically the building $X$ is a
maximal unramified cover of $X_\G$ with covering group $\G$. The fibre of an $i$-simplex of $X_\G$ represented by $s \in S_i$ is $\G s$. For the convenience of later discussions, we require that if a pointed $2$-simplex $[\pi^{-1}a_2 \supsetneq a_0 \supsetneq a_1 \supsetneq a_2]$ lies in $S_2$, so do  $[a_0 \supsetneq a_1 \supsetneq a_2 \supsetneq \pi a_0]$ and $[a_1 \supsetneq a_2 \supsetneq \pi a_0 \supsetneq \pi a_1]$; and if a pointed $1$-simplex $[\pi^{-1}a_2 \supsetneq a_1 \supsetneq a_2]$ lies in $S_1$, then so does its opposite $[a_1 \supsetneq a_2 \supsetneq \pi a_1]$. For $i = 0, 1, 2$, the cardinality of $S_i$ is $(i+1)N_i$, where $N_i$ is the number of pointed $i$-simplices in $X_\G$.

\subsection{Cochain groups}

Let $V_\rho[u]$ denote the tensor product $V_\rho \otimes_{\mathbb C} \mathbb C[u]$ of $V_\rho$ with the polynomial ring $\mathbb C[u]$. It is a free $\mathbb C[u]$-module of rank $d$ admitting the action by $\G$ on $V_\rho$.
For each $i \in \{0, 1, 2\}$ denote by $\dot{C}^i(X_\G, \rho) = \dot{C}^i(X_\G, V_\rho[u])$ the space 
\begin{eqnarray*}
\dot{C}^i(X_\G, \rho)
&=&  \{ f: X_i \to V_\rho[u] \, | \, f(\g x_i) = \rho(\g) f(x_i) \, \rm{for~all~} \g \in \G \rm{~and~} x_i \in X_i\}.
\end{eqnarray*}
Note that functions in $\dot{C}^i(X_\G, \rho)$ are determined by their values on $S_i$, hence it is a free module over $\mathbb C[u]$ of rank $d(i+1)N_i.$

\subsection{Vertex Adjacency operators} \label{adjacencyoperator1}

Let $A_1$ and $A_2$ be the vertex adjacency operators on $\dot{C}^0(X_\G, \rho)$ given by
$$ A_i f([a_0]) = \sum_{a_0 \supsetneq b \supsetneq \pi a_0, |a_0/b|=i}f([b]).$$

Since the vertices of $X$ can be parametrized by the cosets $G/K$, functions in $\dot{C}^0(X_\G, \rho)$ as described above are precisely the functions in the space $\Ind_\G^G \rho$ which are right invariant by $K$. Therefore we may identify $\dot{C}^0(X_\G, \rho)$ with the set {$(\Ind_\G^G \rho)^K \otimes \C[u]$.}
Under this identification, the adjacency operators $A_1$ and $A_2$ defined above coincide with the Hecke operators $A_1$ and $A_2$ in \S \ref{lfunction}.
In view of Proposition \ref{Lfunction1}, we conclude
\begin{proposition} \label{Lfunction}
$$ L( \Ind_\G^G \rho, qu ) = \frac{1}{\det(I-A_1 u + q A_2 u^2 - q^3 I u^3 ~|~  \dot{C}^0(X_\Gamma, \rho))}.$$
\end{proposition}

\bigskip

\subsection{Artin $L$-functions attached to representations of $\G$}
For $i = 1, 2$, two closed $i$-dimensional paths in $X_\G$ are called {\it equivalent} if one can be obtained from the other by changing the starting simplex. Denote  by $[c]$ the equivalence class of a closed $i$-dimensional path $c$.
A closed $i$-dimensional path $c$ in $X_\G$ is called
{\it primitive} if it is not obtained by repeating a shorter path more than once; {it is called a uni-type geodesic if it lifts to a uni-type geodesic in $X$.}

Denote by $\mathcal{P}_i^{(n)}$ the set of all $i$-dimensional {uni-type closed geodesics} in $X_\G$ with geometric length $n$, and by $\mathcal{P}_i$ the union of $\mathcal{P}_i^{(n)}$ for $n \ge 1$. Let $\mathcal{P}_i^{pr}$ be the subset of primitive paths in $\mathcal{P}_i$, and
$[\mathcal{P}_i^{pr}]$ be the set of equivalence classes of paths in $\mathcal{P}_i^{pr}$. The elements in $[\mathcal{P}_i^{pr}]$ play the role of primes for the $i$-th zeta and Artin $L$-functions.

Given an element $\mathfrak{p}$ in $\P_1^{(n)}$, its starting pointed edge is represented by a unique $s_0 \in S_1$ and $\mathfrak{p}$ can be uniquely lifted to a uni-type path $\tilde{\mathfrak{p}} : s_0 \to s_1 \to \cdots \to s_n = \g_ \mathfrak{p} s_0$ in $X$, where $\g_ \mathfrak{p} \in \G$. If ${\mathfrak{p}}$ is lifted to a path $ \tilde{\mathfrak{p}}'$ in $X$ starting at $s_0'= \g_0 s_0$, then $ \tilde{\mathfrak{p}}'$ ends at $\g_0 \g_\mathfrak{p} \g_0^{-1}s_0'$. Hence to $\mathfrak{p}$ in $\P_1$ we can associate an element $\g_ \mathfrak{p} \in \G$ which is unique up to conjugation. Note that if $s_0' = \g_\mathfrak{p} s_0$, then $ \tilde{\mathfrak{p}}' = \g_{\mathfrak{p}}\tilde{\mathfrak{p}}$. Thus $\mathfrak{p}$ repeated twice is lifted to the path $s_0 \to s_1 \to \cdots \to s_n = \g_ \mathfrak{p} s_0 \to \g_ \mathfrak{p} s_1 \to \cdots \to \g_ \mathfrak{p}s_n = \g_ \mathfrak{p}^2 s_0$. The projection to $X_\G$ of the sub-paths $s_j \to \cdots \to \g_ \mathfrak{p}s_j$ for $1 \le j \le n-1$ runs through the paths equivalent to $\mathfrak{p}$.
This shows that the conjugacy class $[\g_ \mathfrak{p}]$ in $\G$ of $\g_ \mathfrak{p}$ depends only on the equivalence class $[\mathfrak{p}]$ of $\mathfrak{p}$. When $\mathfrak p$ is primitive, call $[\g_ \mathfrak{p}]$ the \lq\lq Frobenius at the prime $\mathfrak{p}$" and denote it by $Frob_{[\mathfrak p]}$.

In a similar manner, to $\mathfrak{g}$ in $\P_2$ we associate the conjugacy class $[\g_{\mathfrak{g}}]$ and define $Frob_{[\mathfrak g]}$ for each prime $[\mathfrak g] \in [\mathcal{P}_2^{pr}]$. Observe that if $\mathfrak{p}$ in $\P_i$ is obtained from the path $\mathfrak{p}'$ by repeating it $k$-times,
then $l_A(\mathfrak{p})= k\cdot l_A(\mathfrak{p}')$, $l_G(\mathfrak{p})= k \cdot l_G(\mathfrak{p}')$ and $\g_{\mathfrak{p}}=(\g_{\mathfrak{p}'})^k$.

Now fix a $d$-dimensional representation $(\rho,V_\rho)$ of $\Gamma$. For $i=1, 2$ define the $i$-th Artin $L$-function of $X_\G$ associated to $\rho$ to be
\begin{eqnarray}
L_i(X_\G, \rho, u) =   \prod_{[c] \in [\P_i^{Pr}]} \frac{1}{ \det\left( I- \rho(Frob_{[c]}) u^{l_A(c)} \right)}.
\end{eqnarray}
{Note that when $\rho$ is the trivial representation of $\G$, the $i$th Artin $L$-function coincides with the zeta function $Z_i(X_\G, u)$ defined in \cite{KL, KLW}.} Since the determinant of a matrix is invariant under conjugation and two equivalent paths have the same algebraic length, the Artin $L$-function above is well-defined. We shall show in \S4.6 that it converges absolutely for $|u|$ small to the reciprocal of a polynomial.

\subsection{Edge adjacency operator} \label{adjacencyoperator2}
Define the edge adjacency operator $A_E(\rho, u)$ on $\dot{C}^1(X_\G, \rho)$ by sending $f \in \dot{C}^1(X_\G, \rho)$ to $A_E(\rho, u)f$ whose value at $e \in X_1$ is given by
$$ A_E(\rho, u) f(e) = u^{l_A(e)} \sum_{e' \in N(e)} f(e').$$ 
We proceed to represent $A_E(\rho, u)$ by a block matrix $M_E(\rho, u)$ whose rows and columns are parametrized by the set $S_1$ representing the pointed 1-simplices in $X_\G$. Given $s \in S_1$, consider the above definition at $e = s$. Then $e' \in  N(e)$ lies in the $\G$-orbit of some $s'  \in S_1$, and there is a unique $\g_{ss'} \in \G$ such that $e' = \g_{ss'} s'$ and hence $f(e') = \rho(\g_{ss'}) f(s')$. The $ss'$-entry of $M_E(\rho, u)$ is the $d \times d$ matrix $\rho(\g_{ss'}) u^{l_A(s)}$ if $\G s'$ is an out-neighbor of $\G s$ in $X_\G$, and the zero $d \times d$ matrix otherwise.
The block matrix $M_E(\rho, u)$ representing $A_E(\rho, u)$ depends on the choice of the set $S_1$. For a difference choice of $S_1$, the matrix is replaced by a conjugation.

\subsection{A proof of Theorem \ref{z1}}
A closed path $\mathfrak{p}$ in $\P_1^{(n)}$ is a sequence of pointed edges 
$ e_0 \to e_1 \cdots \to e_n = e_0$ of length $n$, where each $e_i$ is represented by a unique $s_i \in S_1$ and $e_{i+1}$ is an out-neighbor of $e_i$ in $X_\G$ for $0 \le i \le n-1$. Its lifting $\tilde{\mathfrak{p}}$ in $X$ starting at $s_0$ is
$$ \tilde{\mathfrak{p}}: s_0 \to \g_{s_0 s_1} s_1 \to  (\g_{s_0 s_1} \g_{s_1 s_2}) s_2 \to \cdots \to (\g_{s_0 s_1}\cdots \g_{s_{n-1}s_n})s_n= \g_ \mathfrak{p} s_0$$
so that the associated $\g_\mathfrak{p}$ explained in \S4.4 is $\g_\mathfrak{p} = \g_{s_0 s_1}\cdots \g_{s_{n-1}s_n}$. Thus we have
$$\rho(\g_\mathfrak{p}) = \rho(\g_{s_0 s_1})\cdots \rho(\g_{s_{n-1}s_n}).$$

\noindent On the other hand, the $s_0 s_0$ entry of $M_E(\rho, u)^n$ is the sum of all possible products of $n$ entries of $M_E(\rho, u)$ of the form $\rho(\g_{s_0 s_1})\rho(\gamma_{s_1s_2})...\rho(\g_{s_{n-1}s_n})u^{l_A(s_0)+l_A(s_1)+\cdots+l_A(s_{n-1})}$, in which $s_n = s_0$, and for $0 \le i \le n-1$, each $\G s_{i+1}$ is an out-neighbor of $\G s_i$. In other words, the $s_0 s_0$ entry of $M_E(\rho, u)^n$ is the sum of $\rho(\g_\mathfrak{p}) u^{l_A(\mathfrak{p})}$ over elements $\mathfrak{p}$ in $\P_1^{(n)}$ starting at $\G s_0$. This shows that
\begin{eqnarray}\label{equaltrace}
  \Tr( M_E(\rho, u)^{n}) = \sum_{\mathfrak{p} \in \P_1^{(n)} } \Tr( \rho(\g_\mathfrak{p})) u^{l_A(\mathfrak{p})}.
 \end{eqnarray}
To proceed, we shall use the following well-known facts in linear algebra.
\begin{proposition}\label{trdet}
Let $A$ be a square matrix over $\C$ with norm less than 1. Then
$$ \log(I-A)= -\sum_{n=1}^\infty \frac{A^n}{n} \,\, \mbox{converges} \qquad \mbox{and} \qquad \Tr(\log(I-A))=\log(\det(I-A)).$$
\end{proposition}

Order the pointed edges in $S_1$ so that those of type $1$ are before those of type $2$. Then $M_E(\rho,u)$ is of the form
$$ M_E(\rho,u) = \begin{pmatrix} B_1 u & \\  & B_2 u^2\end{pmatrix}$$
for some $dN_1\times d N_1$ complex matrices $B_1$ and $B_2$. The norms of $B_1$ and $B_2$ are bounded and depend only on $A_E(\rho, u)$. It follows from Proposition \ref{trdet} and (\ref{equaltrace}) that, for $|u|<\min\{\|B_1\|^{-1}, \sqrt{\|B_2\|^{-1}}\}$, there holds
\begin{eqnarray*}
\log \det(I- A_E(\rho, u)) &=& \log \det(I- M_E(\rho, u))= \Tr(\log(I- M_E(\rho, u))) =
-\sum_{n=1}^{\infty} \frac{\Tr( M_E(\rho, u)^n )}{n}  \\
&=& - \sum_{n=1}^{\infty} \sum_{\mathfrak{p} \in \P_1^{(n)}}  \frac{\Tr( \rho(\g_\mathfrak{p})) u^{l_A(\mathfrak{p})}}{n}
= -\sum_{\mathfrak{p} \in \P_1}  \frac{\Tr( \rho(\g_\mathfrak{p})) u^{l_A(\mathfrak{p})}}{l_G(\mathfrak{p})} \\
&=& -\sum_{m=1}^\infty \sum_{\mathfrak{p} \in \P_1^{Pr}} \frac{ \Tr( \rho(\g_{\mathfrak{p}^m})) u^{l_A(\mathfrak{p}^m)}} {l_G(\mathfrak{p}^m)}
= -\sum_{m=1}^\infty \sum_{\mathfrak{p} \in \P_1^{Pr}} \frac{ \Tr \left( \rho(\g_\mathfrak{p}) u^{l_A(\mathfrak{p})} \right)^m} {m l_G(\mathfrak{p})}.
\end{eqnarray*}
Since the number of closed paths equivalent to a primitive cycle is equal to its geometric length, the above can be rewritten as
\begin{eqnarray*}
\log \det(I- A_E(\rho, u))  &=& - \Tr\left( \sum_{[\mathfrak{p}] \in [\P_1^{Pr}]} \sum_{m=1}^\infty  \frac{ (\rho(Frob_{[\mathfrak{p}]}) u^{l_A(\mathfrak{p})})^m} {m} \right)\\
&=& \Tr\left( \sum_{[\mathfrak{p}] \in [\P_1^{Pr}]}  \log\left(I-\rho(Frob_{[\mathfrak{p}]}) u^{l_A(\mathfrak{p})}\right)\right)\\
&=&  \log\left( \prod_{[\mathfrak{p}] \in [\P_1^{Pr}]}  \det\left(I-\rho(Frob_{[\mathfrak{p}]}) u^{l_A(\mathfrak{p})}\right)\right)= -\log L_1(X_\G,\rho,u).
\end{eqnarray*}
Exponentiating both sides proves Theorem \ref{z1}.

\subsection{Chamber adjacency operator and a proof of Theorem \ref{z2}} \label{adjacencyoperator3}

Define the chamber adjacency operator $A_C(\rho, u)$ on $ \dot{C}^2(X_\G, \rho)$ by sending $f \in \dot{C}^2(X_\G, \rho)$ to $A_C(\rho, u)f$ whose value at $c \in X_2$ is given by
\begin{eqnarray*}
A_C(\rho, u) f(c) = u^{l_A(c)} \sum_{c' \in N(c)} f(c') = u \sum_{c' \in N(c)} f(c')
\end{eqnarray*}
because all pointed $2$-simplices have algebraic length equal to $1$. Similar to the edge adjacency operator, the chamber adjacency operator $A_C(\rho, u)$ can be represented by a block matrix $M_C(\rho, u)$  whose rows and columns are parametrized by the set $S_2$ representing the pointed 2-simplices in $X_\G$. The $cc'$ entry of $M_C(\rho, u)$ is the $d \times d$ matrix $\rho(\g_{c c'})u$ if $\G c'$ is an out-neighbor of $\G c$ in $X_\G$, and the zero $d \times d$ matrix otherwise.

By an argument similar to the previous subsection, we have that, for $|u|$ small enough,
\begin{eqnarray*}
\log \det(I+ A_C(\rho, u))  &=&  -\log L_2(X_\G,\rho,-u),
\end{eqnarray*}
and hence Theorem \ref{z2} holds.

\section{Pointed simplicial cohomology}
\subsection{Pointed simplicial cohomology groups}
For $i = 0, 1, 2$ denote by $\dot{C}^i(X)$ the free $\mathbb C[u]$-module of functions $f_i: X_i \to V_\rho[u]$. The action of $\G$ on $X_i$ yields the action of $\G$ on $\dot{C}^i(X)$ given by $(\g f_i)(x) = f_i(\g x)$ for all $\g \in \G$, $f_i \in C^i(X)$ and $x \in X_i$. Then $\dot{C}^i(X_\G, \rho)$ consists of the functions in $\dot{C}^i(X)$ on which the action of $\G$ is given by $\rho$.
Define the map $d_i = d_i(u): \dot{C}^i(X)\to \dot{C}^{i+1}(X)$ by
$$ (d_0f_0)([ \pi^{-1} a_1 \supsetneq a_0 \supsetneq a_1])=u^{|a_0/a_1|} f_0([a_1])- f_0([a_0]) , $$
$$ (d_1f_1)([\pi^{-1}a_2 \supsetneq a_0 \supsetneq a_1 \supsetneq a_2 ])=u f_1([\pi^{-1}a_2 \supsetneq a_1 \supsetneq a_2])- f_1([\pi^{-1} a_2 \supsetneq a_0 \supsetneq a_2])+f_1([\pi^{-1} a_1 \supsetneq a_0 \supsetneq a_1]),$$
and all other $d_j$ to be the zero map. It follows from
\begin{align*}
& d_1(d_0f_0)([\pi^{-1}a_2 \supsetneq a_0 \supsetneq a_1 \supsetneq a_2 ])\\
&= u \cdot d_0f_0([\pi^{-1}a_2 \supsetneq a_1 \supsetneq a_2])- d_0f_0([\pi^{-1} a_2 \supsetneq a_0 \supsetneq a_2])+d_0f_0([\pi^{-1} a_1 \supsetneq a_0 \supsetneq a_1]) \\
&=u\left(uf_0([a_2])- f_0([a_1])\right)- \left(u^2 f_0([a_2])- f_0([a_0])\right)+ \left(uf_0([a_1])- f_0([a_0])\right) =0
\end{align*}
 that the $d_i$'s are coboundary  maps. Note that when $u=1$, $d_0$ and $d_1$ are the usual coboundary maps.  As $d_i$ commutes with the action of $\G$, it defines a coboundary map $d_i : \dot{C}^i(X_\G, \rho) \to \dot{C}^{i+1}(X_\G, \rho)$. This gives rise to the $i$-th pointed simplicial cohomology group
$$ \dot{H}^i(X_\G, \rho)= \ker(d_{i}) / \mbox{Im}(d_{i-1}) \qquad {\rm for}~ i = 0, 1, 2,$$
which measures the failure of exactness at $C^i(X_\G, \rho)$ of the cochain complex $$ C^* : 0 {\to} C^0(X_\G, \rho) \stackrel{d_0}{\to} C^1(X_\G, \rho) \stackrel{d_1}{\to} C^2(X_\G, \rho) {\to} 0.$$ \

For $i = 1, 2$ define the map
$\delta_i = \delta_i(u): \dot{C}^i(X)\to \dot{C}^{i-1}(X)$ which sends $f_i \in  \dot{C}^i(X)$ to $\dot{C}^{i-1}(X)$ given by
\begin{eqnarray*}
(\delta_1 f_1)([a_0])&=& \sum_{ a_0 \supsetneq b \supsetneq \pi a_0, |a_0/b|=1} u f_1([ a_0 \supsetneq b \supsetneq \pi a_0])  -\sum_{ a_0 \supsetneq b \supsetneq \pi a_0, |a_0/b|=2} qu^2 f_1([a_0 \supsetneq b \supsetneq \pi a_0]) \\
&=&  \sum_{a_0 \supsetneq b \supsetneq \pi a_0} (-q)^{|a_0/b|-1} u^{|a_0/b|}f_1([ a_0 \supsetneq b \supsetneq \pi a_0])
\end{eqnarray*}
and
$$ (\delta_2 f_2)([\pi^{-1}a_1 \supsetneq a_0 \supsetneq a_1]) = \sum_{a_0 \supsetneq b\supsetneq a_1}  -u f_2([a_0 \supsetneq b \supsetneq a_1 \supsetneq \pi a_0])
+ \sum_{a_1 \supsetneq b \supsetneq \pi a_0}  u^2 f_2([a_1 \supsetneq b \supsetneq \pi a_0 \supsetneq \pi a_1]).$$

\noindent Note that in $\delta_2f_2$ only the first or the second sum is nonempty according as $|a_0/a_1| = 2$ or $1$.
Since $\delta_i$ commutes with the action of $\G$, it defines a map  $\delta_i(u) : \dot{C}^i(X_\G, \rho) \to \dot{C}^{i-1}(X_\G, \rho)$ .

Let
$$\Delta_0(u) = \delta_1(u) d_0(u), \quad \Delta_1(u)= \delta_2(u)  d_1(u) + d_0(u)  \delta_1(u) \qquad \mbox{and}\qquad \Delta_2(u) = d_1(u) \delta_2(u).$$
Observe that $\Delta_i(u)$ is a cochain endomorphism on $\dot{C}^i(X_\G, \rho)$ and
$$\Delta_i(u) \equiv 0 \quad \mbox{on } \dot{H}^i(X_\G, \rho).$$
For $i=0,1$ and 2, define
$$ \Phi_i(u) = \Delta_i(u) + (1- u^3)I,$$
which is also a cochain endomorphism on $\dot{C}^i(X_\G, \rho)$.

Assuming Theorem \ref{thm2}, which will be proved in the next section, we establish Theorem \ref{thm1}. By setting $u=0$, it is obvious that $\Phi_i(u)$ on $\dot{C}^i(X_\G, \rho)$ has nonzero determinant  for $i = 0, 1, 2$. By a general theory on cohomology groups,
we have
\begin{eqnarray*}
\prod_{i=0}^{2}\det(\Phi_i(u) ~|~ \dot{C}^i(X_\G, \rho))^{(-1)^i}
&=& \prod_{i=0}^{2}\det(\Phi_i(u) ~|~  \dot{H}^i(X_\G, \rho))^{(-1)^i}=\prod_{i=0}^{2}\det( (1-u^3)I ~|~ \dot{H}^i(X_\G, \rho))^{(-1)^i}\\
&=& \prod_{i=0}^{2}\det( (1-u^3)I ~|~  \dot{C}^i(X_\G, \rho))^{(-1)^i}=(1-u^3)^{d(N_0-2N_1+3N_2)}.\\
\end{eqnarray*}
On the other hand, by Theorem \ref{thm2}, we also have
\begin{eqnarray*}
\prod_{i=0}^{2}\det(\Phi_i(u) ~|~  \dot{C}^i(X_\G, \rho))^{(-1)^i}
&=& \frac{\det(I-A_1(\rho) u + q A_2 (\rho)u^2 -q^3 u^3 I) (1-u^3)^{2dN_2}\det\left(I+A_C(\rho, u)\right) }{(1-u^3)^{d N_1}\det\left(I-A_E(\rho, u)\right)}\\
&=& (1-u^3)^{d(2N_2-N_1)} \frac{L_1(X_\G,\rho, u)}{L(\Ind^G_\G \rho,qu) L_2(X_\G, \rho, -u)}.
\end{eqnarray*}

\noindent Comparing the above two expressions of the alternating product of $\det(\Phi_i(u))$, we obtain
$$ (1-u^3)^{\chi(X_\G)d} L(\Ind^G_\G \rho, qu) =  \frac{L_1(X_\G,\rho, u)}{L_2(X_\G, \rho, -u)},$$
which is Theorem \ref{thm1}.

\section{A cohomological proof of Theorem \ref{thm2}}

To ease our notation, we shall write $A_E$ for $A_E(\rho, u)$, $A_C$ for $A_C(\rho, u)$, $d_i$ for $d_i(u)$, $\delta_i$ for $\delta_i(u)$, $\Delta_i$ for $\Delta_i(u)$, and $\Phi_i$ for $\Phi_i(u)$ when this will not cause any confusion.

\subsection{The operator $\Phi_0(u)$}

Recall that for a lattice $a_0$, $a_0/\pi a_0 \cong \mathbb{F}_q^3$. In this $3$-dimensional vector space over $\mathbb F_q$ there are $q^2 + q + 1$ lines and the same number of planes. Further, a line  in this space is contained in $q+1$ planes. These results are restated in terms of lattices as follows.

\begin{proposition} \label{counting} $ $\\
(a) Given a lattice $a_0$, the number of pointed edges $[\pi^{-1}b \supsetneq a_0 \supsetneq b]$ in $X_1$ of type $i$ is equal to $q^2+q+1$ for $i= 1$ or $2$.\\
(b) Given a type $1$ pointed edge $[a_0 \supsetneq c \supsetneq \pi a_0]$, there are $q+1$ pointed chambers of the form $[a_0 \supsetneq b \supsetneq c \supsetneq \pi a_0]$.
\end{proposition}

Using Proposition \ref{counting}, (a), we compute $\Phi_0 f$ for $f \in \dot{C}^0(X)$:
\begin{eqnarray*}
(\Phi_0(u) f)([a_0]) &= &(\delta_{1}d_{0}+1-u^3)f([a_0])  \\
&=& (1-u^3)f([a_0]) + \sum_{a_0 \supsetneq b \supsetneq \pi a_0} (-q)^{|a_0/b|-1}u^{|a_0/b|}(d_0f)([a_0 \supsetneq b \supsetneq \pi a_0]) \\
&=& (1-u^3)f([a_0])  + \sum_{a_0 \supsetneq b \supsetneq \pi a_0} (-q)^{|a_0/b|-1}u^{|a_0/b|}\left(u^{|b/\pi a_0|} f([\pi a_0]) - f([b])\right) \\
&=& (1-u^3)f([a_0]) + (q^2 + q + 1 -q(q^2 + q + 1))u^3f([a_0]) - u A_1f([a_0]) + qu^2 (A_2f)([a_0]) \\
&=& (I-A_1 u + q A_2 u^2 - q^3 u^3I) f([a_0]).
\end{eqnarray*} In other words, $\Phi_0(u) = I-A_1 u + q A_2 u^2 - q^3 u^3I$ and hence  $\det(\Phi_0(u) ~|~ C^0(X_\G, \rho)) = L(\Ind_\G^G \rho , qu)^{-1}$ by Proposition 4.3.1. This proves Theorem \ref{thm2}, (1).

\subsection{The operator $\Phi_1(u)$}
Introduce the following two operators on $\dot{C}^1(X)$:
\begin{eqnarray*}
Q(u)f([\pi^{-1}a_1 \supsetneq a_0 \supsetneq a_1]) = \sum_{a_0 \supseteq b \supsetneq a_1} u^{|a_0/b|} f([\pi^{-1}a_1 \supsetneq b \supsetneq a_1])
\end{eqnarray*}
and
\begin{eqnarray*}
J_E(u)f([\pi^{-1}a_1 \supsetneq a_0 \supsetneq a_1]) &=& u^{|a_0/a_1|} f([a_0 \supsetneq a_1 \supsetneq \pi a_0]).
\end{eqnarray*}
They preserve the subspace $\dot{C}^1(X_\G, \rho)$, and will be viewed as operators on this space. As such, $J_E = J_E(u)$ is an involution up to scalar, more precisely,
$J_E^2$ is multiplication by $u^3$. A straightforward computation shows that
\begin{eqnarray}\label{AE}
\qquad J_EA_EJ_E^{-1}f([\pi^{-1}a_1 \supsetneq a_0 \supsetneq a_1]) =
u^{|a_1/\pi a_0|} \sum_{a_0 \supsetneq b \supsetneq \pi a_0, |b/\pi a_0| = |a_0/a_1|, a_1 \not\supset b, b \not\supset a_1}f([a_0 \supsetneq b \supsetneq \pi a_0]).
\end{eqnarray}
Furthermore, $Q = Q(u)$ is unipotent with determinant $1$ on $\dot{C}^1(X_\G, \rho)$, hence it is an automorphism there. Under $Q$, the action of $\delta_1$ is much simplified. More precisely, for $f_1 \in \dot{C}^1(X_\G, \rho)$ we have
\begin{eqnarray} \label{simplification-2}
\qquad (\delta_1  Q) f_1([a_0])&=&  \sum_{a_0 \supsetneq b \supsetneq \pi a_0, |a_0/b|=1} u Qf_1([a_0 \supsetneq b \supsetneq \pi a_0])-\sum_{a_0 \supsetneq b \supsetneq \pi a_0,|a_0/b|=2} qu^2 Qf_1([a_0 \supsetneq b \supsetneq \pi a_0]) \\
 &=& \sum_{a_0 \supsetneq b \supsetneq \pi a_0, |a_0/b|=1} u^{|a_0/b|}\sum_{b \supseteq c \supsetneq \pi a_0} u^{|b/c|} f_1([a_0 \supsetneq c \supsetneq \pi a_0])\nonumber\\
 & & -
 \sum_{a_0 \supsetneq b \supsetneq \pi a_0,|a_0/b|=2} qu^2 f_1([a_0 \supsetneq b \supsetneq \pi a_0]) \nonumber \\
 &=& \sum_{a_0 \supsetneq c \supsetneq \pi a_0, |a_0/c|=1} u f_1([a_0 \supsetneq c \supsetneq \pi a_0])+\sum_{a_0 \supsetneq b \supsetneq \pi a_0, |a_0/c|=2} (q+1)u^2 f_1([a_0 \supsetneq c \supsetneq \pi a_0]) \nonumber \\
 & &-\sum_{a_0 \supsetneq b \supsetneq \pi a_0,|a_0/b|=2} qu^2 f_1([a_0 \supsetneq b \supsetneq  \pi a_0]) \nonumber \\
  &=&  \sum_{a_0 \supsetneq c \supsetneq \pi a_0}  u^{|a_0/c|}f_1([a_0 \supsetneq c \supsetneq \pi a_0]). \nonumber
\end{eqnarray}

The operator $$W(u) = I + J_E(u)$$ also simplifies our computations. Write $W$ for $W(u)$ for short. Indeed, for $f_0 \in \dot{C}^0(X_\G, \rho)$, we have

\begin{eqnarray} \label{simplification-1}
(Wd_0) f_0([\pi^{-1}a_1 \supsetneq a_0 \supsetneq a_1]) &=& d_0 f_0([\pi^{-1}a_1 \supsetneq a_0 \supsetneq a_1])+u^{|a_0/a_1|}d_0 f_0([a_0 \supsetneq a_1 \supsetneq \pi a_0]) \\
 &=& u^{|a_0/a_1|}f_0([a_1])-f_0([a_0])+u^{|a_0/a_1|}( u^{|a_1/\pi a_0|}f_0([a_0])-f_0([a_1])) \nonumber\\
 &=& -(1-u^3)f_0([a_0]). \nonumber
\end{eqnarray}
Further, for $f_2 \in \dot{C}^2(X_\G, \rho)$ and a pointed edge $[\pi^{-1}a_1 \supsetneq a_0 \supsetneq a_1] \in S_1$ with $|a_0/a_1|=1$, we have
\begin{eqnarray} \label{simplification-3}
(W \delta_2) f_2([\pi^{-1}a_1 \supsetneq a_0 \supsetneq a_1]) &=& \delta_2 f_2([\pi^{-1}a_1 \supsetneq a_0 \supsetneq a_1])-u\delta_2 f_2([a_0 \supsetneq a_1 \supsetneq \pi a_0]) =0
\end{eqnarray}
and
\begin{eqnarray} \label{simplification-4}
(W \delta_2) f_2([a_0 \supsetneq a_1 \supsetneq \pi a_0]) &=& \delta_2 f_2([a_0 \supsetneq a_1 \supsetneq \pi a_0])-u^2\delta_2 f_2([\pi^{-1}a_1 \supsetneq a_0 \supsetneq a_1]) \\
&=& (1-u^3) \delta_2 f_2([a_0 \supsetneq a_1 \supsetneq \pi a_0]). \nonumber
\end{eqnarray}

Given $f \in C^1(X_\G, \rho)$ we apply the above results to compute $\frac{1}{1-u^3}(W \Phi_1 Q) f([\pi^{-1}a_1 \supsetneq a_0 \supsetneq a_1])$ according to the type $|a_0/a_1|$ of the pointed edge.
\bigskip

\noindent Case I. $|a_0/a_1| = 1$. By (\ref{AE})-(\ref{simplification-3}) we have
\begin{eqnarray}\label{eq1}
& & \frac{1}{1-u^3}(W \Phi_1 Q) f([\pi^{-1}a_1 \supsetneq a_0 \supsetneq a_1]) \\&=& \frac{1}{1-u^3}W((1-u^3)I+\delta_2  d_1 + d_0 \delta_1)Qf([\pi^{-1}a_1 \supsetneq a_0 \supsetneq a_1]) \nonumber \\
&\stackrel{{\rm by} (\ref{simplification-3}),(\ref{simplification-1})}{=}& WQf([\pi^{-1}a_1 \supsetneq a_0 \supsetneq a_1])- \delta_1 Q f([a_0]) \nonumber \\
&\stackrel{{\rm by} (\ref{simplification-2})}{=}& f([\pi^{-1}a_1 \supsetneq a_0 \supsetneq a_1])+ u\sum_{a_1 \supseteq b \supset \pi a_0} u^{|a_1/b|} f([a_0 \supsetneq b \supsetneq \pi a_0]) - \sum_{a_0 \supsetneq c \supsetneq \pi a_0}  u^{|a_0/c|}f([a_0 \supsetneq c \supsetneq \pi a_0]) \nonumber \\
&=& f([\pi^{-1}a_1 \supsetneq a_0 \supsetneq a_1])- \sum_{a_0 \supseteq b \supset \pi a_0, a_1 \not\supseteq b} u^{|a_0/b|} f([a_0 \supsetneq b \supsetneq \pi a_0])\nonumber \\
&\stackrel{{\rm by} (\ref{AE})}{=}& (I - J_EA_EJ_E^{-1})f([\pi^{-1}a_1 \supsetneq a_0 \supsetneq a_1])-\sum_{a_0 \supseteq b \supset \pi a_0, |a_0/b| = 1, b \ne a_1} u f([a_0 \supsetneq b \supsetneq \pi a_0])\nonumber.
\end{eqnarray}

\noindent Case II. $|a_0/a_1|=2$. Applying (\ref{simplification-1}) and (\ref{simplification-2}), we obtain
\begin{eqnarray}  \label{simplification-6}
\frac{1}{1-u^3}(W d_0 \delta_1 Q)f([\pi^{-1}a_1 \supsetneq a_0 \supsetneq a_1])=-\sum_{a_0 \supsetneq c \supsetneq \pi a_0}  u^{|a_0/c|}f([a_0 \supsetneq c \supsetneq \pi a_0]).
\end{eqnarray}
On the other hand,
\begin{eqnarray}\label{simplification-7}
& & \frac{1}{1-u^3}(W \delta_2 d_1 Q)f([\pi^{-1} a_1 \supsetneq a_0 \supsetneq a_1])\\
& \stackrel{{\rm by} (\ref{simplification-4})}{=}& (\delta_2 d_1 Q)f([\pi^{-1} a_1 \supsetneq a_0 \supsetneq a_1])\nonumber \\
&=& \sum_{a_0 \supsetneq b\supsetneq a_1} -u (d_1 Q) f([a_0 \supsetneq b \supsetneq a_1 \supsetneq \pi a_0])\nonumber \\
&=& \sum_{a_0 \supsetneq b\supsetneq a_1}  -u^2 Q f([a_0 \supsetneq a_1 \supsetneq \pi a_0])+ u Q f([a_0 \supsetneq b \supsetneq \pi a_0]) - u Qf([\pi^{-1}a_1 \supsetneq b \supsetneq a_1])\nonumber \\
&=& \sum_{a_0 \supsetneq b\supsetneq a_1} \left( -u^2 f([a_0 \supsetneq a_1 \supsetneq \pi a_0])+  u \sum_{b \supseteq c \supset \pi a_0} u^{|b/c|} f([a_0 \supsetneq c \supsetneq \pi a_0])  - u f([\pi^{-1} a_1 \supsetneq b \supsetneq a_1])\right)\nonumber \\
&=& \sum_{a_0 \supsetneq b\supsetneq a_1} \left(   \sum_{b \supseteq c \supsetneq \pi a_0, c \neq a_1} u^{|a_0/c|} f([a_0 \supsetneq c \supsetneq \pi a_0])\right) -
 \sum_{a_0 \supsetneq b\supsetneq a_1} u^{|a_0/b|} f([\pi^{-1}a_1 \supsetneq b \supsetneq a_1]). \nonumber
\end{eqnarray}

Observe that as $b$ varies among the $q+1$ lattices satisfying $a_0 \supsetneq b\supsetneq a_1$, the sublattices $c$ satisfying $b \supsetneq c \supsetneq \pi a_0$ and $c \ne a_1$ are all distinct since $c + a_1 = b$, so such $c$ runs through the $q^2+q$ sublattices satisfying $a_0 \supsetneq c \supsetneq \pi a_0$ with $|c/\pi a_0| = 1$ and $c \ne a_1$. Adding the above two formulae yields $\frac{1}{1-u^3}(W \Delta_1 Q) f([\pi^{-1}a_1 \supsetneq a_0 \supsetneq a_1])$ on the left side; on the right side the sum over $c$ in (\ref{simplification-6}) and (\ref{simplification-7}) will cancel except for $c = a_1$ and those $c$ satisfying $|a_0/c| = 1$ and $c \not \supset a_1$ (and automatically $a_1 \not \supset c$). This gives
\begin{eqnarray*}
& &\frac{1}{1-u^3}(W \Delta_1 Q) f([\pi^{-1}a_1 \supsetneq a_0 \supsetneq a_1]) \\
&=& -\sum_{a_0 \supsetneq c \supsetneq \pi a_0, |a_0/c| = 1, c \not \supset a_1, a_1 \not \supset c}  u f([a_0 \supsetneq c \supsetneq \pi a_0]) - u^2 f([a_0 \supsetneq a_1 \supsetneq \pi a_0]) - u\sum_{a_0 \supsetneq b\supsetneq a_1} f([\pi^{-1}a_1 \supsetneq b \supsetneq a_1])\\
&=& -(J_EA_EJ_E^{-1})f([\pi^{-1}a_1 \supsetneq a_0 \supsetneq a_1]) - J_EQ f([\pi^{-1}a_1 \supsetneq a_0 \supsetneq a_1])- (Q-I)f([\pi^{-1}a_1 \supsetneq a_0 \supsetneq a_1]).
\end{eqnarray*}
As $\Phi_1 = \Delta_1 + (1 - u^3)I$, the above can be rewritten as
\begin{eqnarray}\label{eq2}
\frac{1}{1-u^3}(W \Phi_1 Q )f([\pi^{-1}a_1 \supsetneq a_0 \supsetneq a_1])
= (I-J_EA_EJ_E^{-1})f([\pi^{-1}a_1 \supsetneq a_0 \supsetneq a_1]).
\end{eqnarray}

Combining (\ref{eq1}) and (\ref{eq2}) yields the following identity on operators
\begin{eqnarray}\label{eq3}
\frac{1}{1-u^3} W \Phi_1 Q = I-J_EA_EJ_E^{-1} - N,
\end{eqnarray}
where $N$ is the operator on $C^1(X_\G, \rho)$ sending $f$ to $Nf$ which is zero at pointed type $2$ edges, and whose value at a pointed type $1$ edge
$[\pi^{-1}a_1 \supsetneq a_0 \supsetneq a_1]$ is given by
$$ Nf([\pi^{-1}a_1 \supsetneq a_0 \supsetneq a_1]) = \sum_{a_0 \supseteq b \supset \pi a_0, |b/\pi a_0| = 2, b \ne a_1} u f([a_0 \supsetneq b \supsetneq \pi a_0]).$$
Thus $N^2 = 0$. As noted before, $(I - J_E)(I + J_E) = I - J_E^2 = (1 - u^3)I$. Hence multiplying both sides of (\ref{eq3}) by $I - J_E$ on the left gives rise to the identity
$$ \Phi_1 = (I - J_E)(I-J_EA_EJ_E^{-1} - N)Q^{-1}.$$

Now we express the determinant of $\Phi_1$ on $C^1(X_\G, \rho)$ in terms of the determinants of the operators on the right hand side on the same space. 
As remarked before $Q$ and hence $Q^{-1}$ have determinant $1$. By pairing off a type $1$ pointed edge $[\pi^{-1} a_1 \supsetneq a_0 \supsetneq a_1] \in S_1$ with its type $2$ opposite $[a_0 \supsetneq a_1 \supsetneq \pi a_0] \in S_1$, we partition the $2N_1$ pointed edges in $S_1$ into $N_1$ pairs and with respect to this basis the operator $I - J_E$ is represented by $N_1$ diagonal block matrices of the form $\begin{pmatrix}I_d & -u^2I_d\\ -uI_d & I_d\end{pmatrix}$, where $I_d$ denotes the identity $d \times d$ matrix. Therefore $\det(I - J_E) = (1 - u^3)^{dN_1}$. Finally to compute the determinant of $I-J_EA_EJ_E^{-1} - N$, we order the pointed edges in $S_1$ by first selecting those of type $1$ then followed by those of type $2$. With respect to this basis, the operator $J_EA_EJ_E^{-1}$ is represented by the diagonal block matrix $\begin{pmatrix} B_1u & 0 \\ 0 & B_2u^2 \end{pmatrix}$ and $J_EA_EJ_E^{-1} + N$ by a lower triangular block matrix $\begin{pmatrix} B_1u & 0 \\ A & B_2u^2 \end{pmatrix}$. Therefore
$$\det(I-J_EA_EJ_E^{-1} - N) = \det(I-J_EA_EJ_E^{-1}) = \det (I - A_E(\rho, u)).$$
Put together, we have shown 
$$\det(\Phi_1(u) ~|~C^1(X_\G, \rho)) = (1 - u^3)^{dN_1}\det (I - A_E(\rho, u))= (1 - u^3)^{dN_1}L_1(X_\G, \rho, u)^{-1},$$
as stated in Theorem \ref{thm2}, (2).

\bigskip

\subsection{The operator $\Phi_2(u)$}

Define the operator $J_C$ on $C^2(X)$ which sends $f \in C^2(X)$ to
$$J_C f ([\pi^{-1}a_2 \supsetneq a_0 \supsetneq a_1 \supsetneq a_2]) =
f([a_0 \supsetneq a_1 \supsetneq a_2 \supsetneq \pi a_0]).$$
It leaves invariant the subspace $\dot C^2(X_\G, \rho)$. Further, for $g \in C^2(X_\G, \rho)$, an easy computation shows
$$J_C A_C J_C^{-1}g([\pi^{-1}a_2 \supsetneq a_0 \supsetneq a_1 \supsetneq a_2]) = \sum_{ a_0 \supsetneq b \supsetneq a_2, ~ b \neq a_1} ug([a_0 \supsetneq b \supsetneq a_2 \supsetneq \pi a_0]).$$

We begin with

\begin{proposition} On $\dot C^2(X_\G, \rho)$ there holds the identity
$$d_1 \delta_2 = (I + J_Cu + J_C^2 u^2)(I + J_CA_CJ_C^{-1}) - (1-u^3)I.$$
Equivalently, $\Phi_2(u) = d_1(u) \delta_2(u) + (1-u^3)I = (I + J_Cu + J_C^2 u^2)(I + J_CA_C(\rho, u)J_C^{-1})$.
\end{proposition}

\begin{proof}We compare both sides evaluated at $f \in \dot C^2(X_\G, \rho)$. The left hand side is
\begin{eqnarray*}
& &d_1 (\delta_2 f)([\pi^{-1}a_2 \supsetneq a_0 \supsetneq a_1 \supsetneq a_2]) \\
&=& u(\delta_2 f)([\pi^{-1}a_2 \supsetneq a_1 \supsetneq a_2]) - (\delta_2 f)([\pi^{-1}a_2 \supsetneq a_0 \supsetneq a_2]) + (\delta_2 f)([\pi^{-1}a_1 \supsetneq a_0 \supsetneq a_1]) \\
&=& u^3\sum_{a_2 \supsetneq b \supsetneq \pi a_1} f([a_2 \supsetneq b \supsetneq \pi a_1 \supsetneq \pi a_2]) + u\sum_{a_0 \supsetneq b \supsetneq a_2} f([a_0 \supsetneq b \supsetneq a_2 \supsetneq \pi a_0])\\
&+& u^2 \sum_{a_1 \supsetneq b \supsetneq \pi a_0} f([a_1 \supsetneq b \supsetneq \pi a_0 \supsetneq \pi a_1]).
\end{eqnarray*}

For the right hand side, we first compute
\begin{eqnarray*}
& &(I + J_Cu + J_C^2 u^2)(I + J_CA_CJ_C^{-1})f([\pi^{-1}a_2 \supsetneq a_0 \supsetneq a_1 \supsetneq a_2]) \\
&=& (I + J_CA_CJ_C^{-1})f([\pi^{-1}a_2 \supsetneq a_0 \supsetneq a_1 \supsetneq a_2]) +
u (I + J_CA_CJ_C^{-1})f([a_0 \supsetneq a_1 \supsetneq a_2 \supsetneq \pi a_0]) + \\
& & u^2 (I + J_CA_CJ_C^{-1})f([ a_1 \supsetneq a_2 \supsetneq \pi a_0 \supsetneq \pi a_1])\\
&=& f([\pi^{-1}a_2 \supsetneq a_0 \supsetneq a_1 \supsetneq a_2]) + u f([a_0 \supsetneq a_1 \supsetneq a_2 \supsetneq \pi a_0]) + u^2 f([ a_1 \supsetneq a_2 \supsetneq \pi a_0 \supsetneq \pi a_1]) \\
& & + u\sum_{a_0 \supsetneq b \supsetneq a_2, b \ne a_1} f([a_0 \supsetneq b \supsetneq a_2 \supsetneq \pi a_0]) + u^2 \sum_{a_1 \supsetneq b \supsetneq \pi a_0, b \ne a_2} f([a_1 \supsetneq b \supsetneq \pi a_0 \supsetneq \pi a_1]) \\
& & + u^3 \sum_{a_2 \supsetneq b \supsetneq \pi a_1, b \ne \pi a_0}
f([a_2 \supsetneq b \supsetneq \pi a_1 \supsetneq \pi a_2]).
\end{eqnarray*}
Therefore the right hand side is equal to
\begin{eqnarray*}
& &(I + J_Cu + J_C^2 u^2)(I + J_CA_CJ_C^{-1})f([\pi^{-1}a_2 \supsetneq a_0 \supsetneq a_1 \supsetneq a_2]) - (1-u^3)f([\pi^{-1}a_2 \supsetneq a_0 \supsetneq a_1 \supsetneq a_2])\\
&= & u\sum_{a_0 \supsetneq b \supsetneq a_2} f([a_0 \supsetneq b \supsetneq a_2 \supsetneq \pi a_0]) + u^2 \sum_{a_1 \supsetneq b \supsetneq \pi a_0} f([a_1 \supsetneq b \supsetneq \pi a_0 \supsetneq \pi a_1]) \\
& & + u^3 \sum_{a_2 \supsetneq b \supsetneq \pi a_1}
f([a_2 \supsetneq b \supsetneq \pi a_1 \supsetneq \pi a_2])\\
&=& (d_1 \delta_2) f([\pi^{-1}a_2 \supsetneq a_0 \supsetneq a_1 \supsetneq a_2]).
\end{eqnarray*} This proves the proposition.
\end{proof}

Now we compute the determinant of $\Phi_2$.
For $c= [\pi^{-1}a_2 \supsetneq a_0 \supsetneq a_1 \supsetneq a_2]$ in $S_2$, write $c'$ for the pointed chamber $[a_0 \supsetneq a_1 \supsetneq a_2 \supsetneq \pi a_0]$ and $c''$ for $[a_1 \supsetneq a_2 \supsetneq \pi a_0 \supsetneq \pi a_1]$. By our choice of $S_2$, $c, c'$ and $c''$ are all in $S_2$, and the pointed chambers in $S_2$ can be partitioned into $N_2$ disjoint triples $\{c, c', c''\}$. With respect to each triple, the operator $I + J_Cu + J_C^2 u^2$ is represented by the $3d \times 3d$ matrix $\begin{pmatrix}I_d & uI_d & u^2I_d\\ u^2I_d & I_d & uI_d \\ uI_d & u^2I_d & I_d \end{pmatrix}$, which has determinant $(1 - u^3)^{2d}$. Here $I_d$ is the $d \times d$ identity matrix.   Thus $\det (I + J_Cu + J_C^2 u^2 ~|~ C^2(X_\G, \rho)) = ( 1 - u^3)^{2dN_2}$. Combined with $$\det (I + J_CA_C(\rho, u)J_C^{-1}) = \det (I + A_C(\rho, u)) = \det (I - A_C(\rho, -u))= L_2(X_\G, \rho, -u)^{-1},$$ we get  $\det(\Phi_2(u) ~|~ C^2(X_\G, \rho)) = (1 - u^3)^{2dN_2}L_2(X_\G, \rho, -u)^{-1}$, as claimed in Theorem \ref{thm2}, (3).

\subsection{Cohomological interpretation of the proof} The computations in \S6.1-\S6.3 can be rephrased as follows. We define two homomorphisms $\Psi_1 =\{\Psi_{1,i}| i = 0, 1, 2\}$ and $\Psi_2 =\{\Psi_{2,i}| i = 0, 1, 2\}$ from the complex $C^*$ to itself as follows. For the first,
$ \Psi_{1, i} : C^i(X_\G, \rho) \to C^i(X_\G, \rho)$ is multiplication by $1-u^3$ for each $0 \le i \le 2$. It is clear that $\Psi_{1, i+1}d_i = d_i \Psi_{1, i}$ for $i = 0, 1$. Hence $\Psi_1$ is an endomorphism of the complex $C^*$. For the second map, $\Psi_{2, i} : C^i(X_\G, \rho) \to C^i(X_\G, \rho)$ are defined as
\begin{eqnarray*}
\Psi_{2,0} &=& I-A_1 u + q A_2 u^2 - q^3 u^3I \\
\Psi_{2,1} &=& (I - J_E)(I-J_EA_E(\rho, u)J_E^{-1} - N)Q^{-1} \\
\Psi_{2,2} &=& (I + J_Cu + J_C^2 u^2)(I + J_CA_C(\rho, u)J_C^{-1}).
\end{eqnarray*}
That $\Psi_2$ is also an endomorphism of the complex $C^*$ follows from the fact that each $\Psi_{2, i} = \Phi_i$ and $\Phi_i$ have the desired property. Further, the relation $\Psi_{2, i} - \Psi_{1, i} = \Delta_i = d_{i-1} \delta_i + \delta_{i+1} d_i$ for each $i$ shows that $\Psi_1$ and $\Psi_2$ are homotopically equivalent. Therefore Theorem \ref{thm1} holds.

\section{A proof of Theorem \ref{induction}}
For a finite-dimensional representation $(\rho', V_{\rho'})$ of a finite-index subgroup $\G'$ of $\G$,
regard the space $V_\rho$ of the induced representation $\rho = \Ind_{\G'}^{\G} \rho'$ as the set
$$ V_\rho = \{ f: \Gamma \to V_{\rho'}: f(\gamma' \g) = \rho'(\gamma')f(\g) \mbox{ for all $\gamma'\in \Gamma', \g \in \Gamma$}\}$$
on which $\rho$ acts by right translation $\rho(\tilde{\gamma})f(\g)=f(\g \tilde{\gamma}).$ Let $i \in \{1, 2\}$.
Given $g \in \dot{C}^i(X_{\G'},\rho')$, define the function $f_g$ on $X_i$ such that its value at  $x_i \in X_i$ is a function $f_g(x_i) : \G \to V_{\rho'}[u]$ given by
$$f_g(x_i)(\g) := g(\g x_i) \qquad \text{for~all}~ \g \in \G.$$
Then for $\g' \in \G'$, it follows from the definition that
$$f_g(x_i)(\g'\g) = g(\g'\g x_i) = \rho'(\g')g(\g x_i) = \rho'(\g')f_g(x_i)(\g),$$
which shows that $f_g(x_i)$ lies in $V_\rho[u]$. Moreover, for $\tilde{\g} \in \G$,
$$f_g(\tilde \g x_i)(\g) = g(\g \tilde{\g} x_i) = f_g(x_i)(\g \tilde \g) = \rho(\tilde \g)f_g(x_i)(\g)$$
implies that $f_g \in \dot{C}^i(X_{\G},\rho)$. Hence $g \mapsto f_g$ defines a homomorphism $\phi_i : \dot{C}^i(X_{\G'},\rho') \to \dot{C}^i(X_{\G},\rho)$
as $\mathbb C[u]$-modules.
Conversely, for $f \in \dot{C}^i(X_{\G},\rho)$, set
$$ g_f(x_i) := f(x_i)(1) \qquad \text{for~all}~ x_i \in X_i.$$
Then for $\g' \in \G'$,
$$ g_f(\g' x_i) = f(\g' x_i)(1) = \bigg( \rho(\g') f(x_i) \bigg)(1) = f(x_i)(\g') = \rho'(\g') \bigg( f(x_i)(1) \bigg)= \rho'(\g')  g_f(x_i).$$
Thus $g_f$ lies in $\dot{C}^i(X_{\G'},\rho')$.
It is easy to see that $f \mapsto g_f$ defines the inverse map of $\phi_i$ which implies that $\phi_i$ is an isomorphism.

Next we claim that the diagram
$$
\begin{tikzpicture}[every node/.style={midway}]
  \matrix[column sep={10em,between origins}, row sep={5em}] at (0,0) {
    \node(L1) {$\dot{C}^1(X_{\G'},V_{\rho'}[u])$};
    & \node(R1) {$\dot{C}^1(X_{\G},V_{\rho}[u])$}; \\
    \node(L2) {$\dot{C}^1(X_{\G'},V_{\rho'}[u])$};
    & \node (R2) {$\dot{C}^1(X_{\G},V_{\rho}[u])$};\\
  };
  \draw[->] (L1) -- (L2) node[anchor=east]  {$A_E(\rho',u)$};
  \draw[->] (L1) -- (R1) node[anchor=south] {$\phi_1$};
  \draw[->] (R1) -- (R2) node[anchor=west] {$A_E(\rho,u)$};
  \draw[->] (L2) -- (R2) node[anchor=north] {$\phi_1$};
\end{tikzpicture}
$$ commutes. If so, then combined with Theorem \ref{z1}, this gives
$$L_{1}(X_{\G'},\rho', u)= \frac{1}{\det(I- A_E(\rho', u))}= \frac{1}{\det(I- A_E(\rho, u))}=L_{1}(X_\G,\rho, u). $$
To prove the claim, given $g \in \dot{C}^1(X_{\G'},\rho')$, $e \in X_1$, and $\g \in \G$, we examine
\begin{eqnarray*}
f_{A_E(\rho', u)g}(e)(\g)&=& A_E(\rho', u)g(\g e) \\
&=& \sum_{e' \in N(e)} u^{l_A(\g e)}g(\g e') \quad \text{since}~ N(\g e) = \g N(e)\\
&=& \sum_{e' \in N(e)}u^{l_A(e)}f_g(e')(\g) \quad \text{since ~ $\G$~ preserves ~ the~ type ~of ~pointed~ edges} \\
&=& (A_E(\rho, u)f_g)(e)(\g),
\end{eqnarray*}
that is, $f_{A_E(\rho', u)g} = A_E(\rho, u)f_g$, as claimed.

A similar argument proves $ L_2(X_{\G'}, \rho', u) = L_2(X_\G, \rho, u)$.

\end{document}